\documentclass[reqno]{amsart}
\usepackage{amsmath}
\usepackage{amsfonts}
\usepackage{amstext}
\usepackage{amsbsy}
\usepackage[english]{babel}
\usepackage{amsopn}
\usepackage{amsxtra}
\usepackage{upref}
\usepackage{amsthm}
\usepackage{amsmath}
\usepackage{amssymb}
\usepackage[shortlabels]{enumitem}
\usepackage{bbm}
\usepackage[pdftex]{graphicx}
\usepackage{hyperref}
\usepackage{mathrsfs}
\usepackage{verbatim}
\usepackage{graphicx}
\usepackage{tikz}
\usepackage{euscript}
\usepackage{enumitem}

\numberwithin{equation}{section}

\newtheorem{teo}{Theorem}[section]
\newtheorem*{teo*}{Theorem}
\newtheorem{prop}[teo]{Proposition}
\newtheorem{lema}[teo]{Lemma}
\newtheorem{coro}[teo]{Corollary}

\newtheorem{defi}[teo]{Definition}

\newtheorem{que}{Question}
\newtheorem{assu}{Assumption}

\newtheorem*{claim*}{Claim}

\DeclareMathSymbol{\varnothing}{\mathord}{AMSb}{"3F}

\def\F{{\mathcal{F}}}

\def\E{{\mathbb E}}
\def\R{{\mathbb R}}

\def\N{{\mathbb N}}

\DeclareMathOperator*{\interior}{int}
\DeclareMathOperator*{\var}{var}
\DeclareMathOperator*{\esssup}{ess\,sup}

\DeclareMathOperator*{\diam}{diam}

\title{Dimension of Gibbs measures with infinite entropy}
\date{\today}

\author{Felipe P\'erez} \address{Felipe P\'erez: School of Mathematics, University of Bristol, University Walk, Clifton, Bristol, BS8 1TW, UK}
\email{\href{mailto:fp16987@bristol.ac.uk}{fp16987@bristol.ac.uk}}

\begin{document}

\begin{abstract}
We study the Hausdorff dimension of Gibbs measures with infinite entropy with respect to maps of the interval with countably many branches. We show that under simple conditions, such measures are symbolic-exact dimensional, and provide an almost sure value for the symbolic dimension. We also show that the lower local dimension dimension is almost surely equal to zero, while the upper local dimension is almost surely equal to the symbolic dimension. In particular, we prove that a large class of Gibbs measures with infinite entropy for the Gauss map have Hausdorff dimension zero and packing dimension equal to $1/2$, and so such measures are not exact dimensional.
\end{abstract}

\maketitle

\section{Introduction}
In this paper we study the dimension of measures invariant under a certain class of maps of the unit interval $[0,1]$: Expanding Markov Renyi (EMR) maps. These maps $T\colon [0,1]\to [0,1]$ admit representations by means of symbolic dynamics, and satisfy smoothness properties that allow us to use ergodic theoretic methods to study their geometric properties. Given an ergodic T-invariant probability measure $\mu$, we are interested in the pointwise behavior of the \emph{local dimension}
\begin{align*}
    d(x)=\lim_{r\to 0}\dfrac{\log \mu(B(x,r))}{\log r}.
\end{align*}
Knowledge of the almost sure behavior of the local dimension yields information about the \emph{Hausdorff} and the \emph{packing} dimension of the measure. There are two dynamical quantities which are particularly relevant when studying the local dimension of such measures: the \emph{metric entropy} $h_\mu$ (or simply the entropy) and the \emph{Lyapunov exponent} $\lambda_\mu$ of $(T,\mu)$. The connection between the entropy and the Lyapunov exponent and the local dimension is well understood when the entropy is finite. Our goal is to investigate the case when both of these quantities are infinite.

Formulae relating the dynamical invariants $h_\mu,\lambda_\mu$ and the local dimension have been extensively studied for the last few decades in the case $h_\mu<\infty$. For Bernoulli measures invariant under the Gauss map, Kinney and Pitcher proved in \cite{kinney1966dimension} that if the measure $\mu$ is defined by a probability vector $p = \{p_i\}$, the Hausdorff dimension of $\mu$ can be computed with the formula
\begin{align*}
    \dim_H \mu = \dfrac{-\sum_{n=1}^\infty p_n\log p_n}{2\int_0^1 |\log x|\mathrm{d}\mu(x)}
\end{align*}
provided that $\sum_{n=1}^\infty p_n\log n<\infty$. In \cite{ledrappier1985dimension} the authors proved that for a $\mathcal{C}^1$ map $T\colon [0,1]\to [0,1]$ where $T$ and $T'$ are piecewise monotonic and the Lyapunov exponent $\lambda_\mu$ is positive, if $\mu$ is an invariant ergodic probability measure, then we have that
\begin{align*}
\lim_{r\to 0} \dfrac{\log \mu(B(x,r))}{\log r}=\dfrac{h_\mu}{\lambda_\mu}.
\end{align*}
In particular, $\dim_H \mu = h_\mu / \lambda_\mu$.
Other versions of the formula were proved by Young and Hofbauer, Raith in \cite{young1982dimension} and \cite{hofbauer1992hausdorff}, among others. In all of these examples, it is assumed  $0<\lambda_\mu < \infty$. In the context of countable Markov systems, Mauldin and Urbanski proved the following theorem:

\begin{teo}[Volume Lemma, \cite{mauldin2003graph}]\label{volumeLemma}
Let $(X,T)$ be a countable Markov shift coded by the shift in countably many symbols $(\Sigma,\sigma)$. Suppose that $\mu$ is a Borel shift-invariant ergodic probability measure on $\Sigma$ such that at least one of the numbers $H(\mu,\alpha)$ or $\lambda_\mu$ is finite, where $H(\mu,\alpha)$ is the entropy of $\mu$ with respect to the natural partition $\alpha$ in cylinders of $\Sigma$. Then
\begin{align*}
    \dim_H (\mu \circ \pi^{-1}) = \dfrac{h_\mu}{\lambda_\mu},
\end{align*}
where $\pi \colon \Sigma\to X$ is the coding map.
\end{teo}

The coding map can be interpreted as a means to go from the symbolic representation of the dynamics to the geometric space. When the local dimension exists and is constant almost everywhere, we say that the measure $\mu$ is \emph{exact dimensional}.

The case when $\lambda_\mu=0$ was studied by Ledrappier and Misiurewicz in \cite{ledrappier1985dimension}, wherein they constructed a $\mathcal{C}^r$ map of the interval and a non-atomic ergodic invariant measure which has zero Lyapunov exponent and is such that the local dimension does not exist almost everywhere. More precisely, they show that the lower local dimension and upper local dimension are not equal:
\begin{align*}
\underline{d}_\mu(x)\liminf_{r\to 0} \dfrac{\log \mu(B(x,r))}{\log r} < \limsup_{r\to 0} \dfrac{\log \mu(B(x,r))}{\log r} = \overline{d}_\mu(x)
\end{align*}
almost everywhere. For this construction, the authors consider a class of unimodal maps (Feigenbaum's maps).

We investigate the Hausdorff dimensions of invariant ergodic measures for piecewise expanding maps of the interval with countably many branches. In particular, we focus on maps exhibiting similar properties to the Gauss map and measures with infinite entropy and infinite Lyapunov exponent. Our main result is (see next section for the definitions):

\begin{teo*}
Let $T\colon [0,1]\to [0,1]$ be a Gauss-like map and $\mu$ be an infinite entropy Gibbs measure satisfying assumption \ref{assumption} and such that the decay ratio $s$ exists . Then $\underline{d}_\mu(x)=0,\overline{d}_\mu(x)=s$ almost everywhere. 
\end{teo*}

We can also compute the almost sure value of the symbolic dimension. The Gibbs assumption on the measure implies that a certain sequence of observables can be seen as a non-integrable stationary ergodic process and allows us to use some tools of infinite ergodic theory developed by Aaronson. In particular, the pointwise behavior of trimmed sums plays a fundamental role in our arguments. We also prove that the packing dimension of such measures is equal to the decay ratio, and conclude that such systems are not exact dimensional. We remark that the methods used in the context of finite entropy fail, as they rely on the fact that the measure and diameter of the iterates of the natural Markov partition decrease at an exponential rate given by $h_\mu$ and $\lambda_\mu$ respectively, enabling the use of coverings by balls of different scales. To tackle this problem, we make use of more refined coverings of balls, which are capable of detecting the asymptotic interaction between the Gibbs measure and the Lebesgue measure. 

The study of the Hausdorff dimension of sets for which their points have infinite Lyapunov exponent has already been considered: see for instance \cite{fan2010frequency} where the authors compute the Hausdorff dimension of sets with prescribed digits on their continued fraction expansion, or \cite{fishman2014diophantine} where the authors construct a measure invariant under the Gauss map which gives full measure to the Liouville numbers. Since the Liouville numbers are a zero dimensional set, such measure is also zero dimensional. Our result shows that this is the case for a large class of measures.

The dimension of Bernoulli measures for the Gauss map $G$ was studied by Kifer, Peres and Weiss in \cite{kifer2001dimension}, where they show that there is a universal constant $\varepsilon_0>10^{-7}$ so that
\begin{align*}
\dim_H (\mu_p \circ \pi^{-1})\leq 1-\varepsilon_0
\end{align*}
for every Bernoulli measure on the symbolic space coding the Gauss map, where $\pi$ is the coding map. This inequality holds even for the case where the entropy of the measure is infinite. They also show that for an infinite entropy Bernoulli measure $\mu$, the Hausdorff dimension satisfies $\dim_H\mu \leq 1/2$. Their method relies on estimating the dimension of the sets of points for which the frequency of a sequence of digits in their continued fraction expansion differs from the expected value by a certain threshold is uniformly (with respect to the sequence of digits) bounded from 1, and a bound on the dimension of points that lie in unusually short cylinders.
This situation has been recently studied by Jurga and Baker (see \cite{2018arXiv180600841J} and \cite{2018arXiv180207585B}) using different methods. Concretly, in \cite{2018arXiv180600841J}  the author uses ideas of the Hilbert-Birkhoff cone theory and extract information about the dynamics through the transfer operator. On the other hand, in \cite{2018arXiv180207585B}) the authors construct a Bernoulli measure $\mu_{q}$ such that $\sup_p \dim_H\mu_p = \dim_H \mu_q$, where the supremum is taken over all Bernoulli measures. This in conjunction with the Variational Principle (see \cite{walters2000introduction}) yield their result.

The paper is structured as follows. In section \ref{sec:not} we introduce the notation used throughout the paper as well as the main objects of study. We also state the results of the paper. In section \ref{sec:symb} we compute the symbolic dimension and characterize it in terms of the Markov partition. In section \ref{sec:infinite} we study the consequences of $\lambda_\mu=\infty$ at the level of the asymptotic rate of contraction of the cylinders. In sections \ref{sec:haus} and \ref{sec:pack} we prove the results for the Hausdorff and the Packing dimension respectively. We finish the article stating some questions of interest that could not be answered with the methods used in this paper. 

\newpage

\section{Notation and statement of main results} \label{sec:not}

\subsection{The class of maps}
We start introducing the EMR (Expanding-Markov-Renyi) maps of the interval. 
\begin{defi}
We say that a map $T\colon I\to I$ of the interval $I=[0,1]$ is an \emph{EMR map} if there is a countable collection of closed intervals $\{I(n)\}$ (with disjoint interiors $\interior{I(n)}$) such that:
\begin{enumerate}
    \item The map is $\mathcal{C}^2$ on $\bigcup_n \interior{I(n)}$,
    \item Some power of $T$ is uniformly expanding, i.e., there is a positive integer $r$ and a constant $\alpha>1$ such that $|(T^r)'(x)|\geq \alpha$ for all $x\in\bigcup_n \interior{I(n)}$,
    \item The map is Markov and can be coded by a full shift (see next subsection),
    \item The map satisfies Renyi's condition: there is a constant $E>0$ such that
\begin{align*}
\sup_{n\in\N}\sup_{x,y,z\in I_n} \dfrac{|T''(x)|}{|T'(y)||T'(z)|}\leq E,
\end{align*}
\end{enumerate}
\end{defi}

This class of maps was first introduced in \cite{pollicott1999multifractal} in the context the multifractal analysis of the Lyapunov exponent for the Gauss map. Renyi's condition provides good estimates for the Lebesgue measure of the cylinders associated to the Markov structure of the map (see next subsection). For simplicity, we will assume that the maps are orientation preserving (the orientation reversing case only differs in the relative position of the cylinders). The set of branches must accumulate at least at one point, and we assume that it accumulates at exactly one point: we also assume that the branches accumulate on the left endpoint of $I$ (the case when the branches accumulate in the right endpoint of $I$ is analogous). Re-indexing if necessary, we can assume that $I(n+1)< I(n)$ for all $n$. Let $r_n=|I(n)|$. 

\begin{defi}
We say that an EMR map $T$ is a \emph{Gauss-like map} if it satisfies the following conditions:
\begin{enumerate}
\item $r_n>0$ for every $n\in\N$,
\item $r_{n+1}\leq r_n$,
\item $\sum_n r_n= 1$,
\item $0<K\leq r_{n+1}/r_n \leq K' < \infty$ for some constants $K,K'$,
\item $\{r_n\}$ decays polynomially as $n$ goes to infinity (see definition \eqref{decay}).
\end{enumerate}
\end{defi}

We want to keep in mind piecewise linear functions as the main example, as for this class of maps, calculations are simplified. We will also keep in mind the example of the Gauss map.

\begin{center}
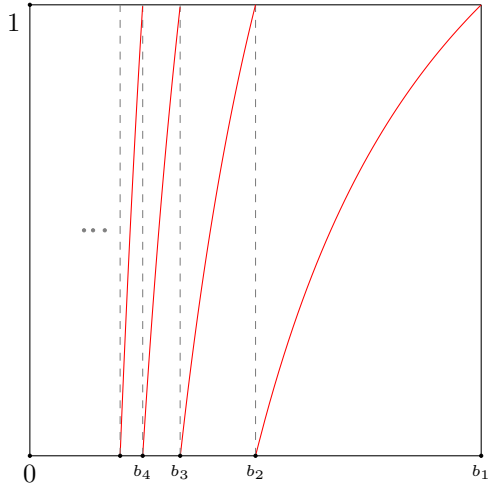
\begin{figure}[h!]
\begin{tikzpicture}[scale=6]
  \draw[-] (0,0) -- (1,0) ;
  \draw[-] (0,0) -- (0,1) ;
  \draw[-] (1,0) -- (1,1) ;
  \draw[-] (0,1) -- (1,1) ;
  \filldraw[black] (0,1) circle (0.1pt) node[below left]{$1$}; 
  \filldraw[black] (0,0) circle (0.1pt) node[below]{$0$}; 
  \filldraw[black] (1,0) circle (0.1pt) node[below]{ \tiny $b_1$}; 
\draw[scale=1,domain=0.5	:1,smooth,variable=\x,red]  plot ({\x},{-1/\x+2  });
\draw[gray,dashed] (0.5,0) -- (0.5,1) ;
    \filldraw[black] (0.5,0) circle (0.1pt) node[below]{\tiny $b_2$}; 
\draw[scale=1,domain=0.333:0.5,smooth,variable=\x,red]  plot ({\x},{-1/\x+3});
  \draw[gray,dashed] (0.333,0) -- (0.333,1) ;
\draw[scale=1,domain=0.25:0.333,smooth,variable=\x,red]  plot ({\x},{-1/\x+4});
  \filldraw[black] (0.333,0) circle (0.1pt) node[below]{\tiny $b_3$}; 
  \draw[gray,dashed] (0.25,0) -- (0.25,1) ;
    \filldraw[black] (0.25,0) circle (0.1pt) node[below]{\tiny $b_4$}; 
\draw[scale=1,domain=0.2:0.25,smooth,variable=\x,red]  plot ({\x},{-1/\x+5});
  \draw[gray,dashed] (0.2,0) -- (0.2,1) ;
  \filldraw[black] (0.2,0) circle (0.1pt) node[below]{}; 
\filldraw[gray] (0.166,0.5) circle (0.1pt) node[anchor=west]{};
\filldraw[gray] (0.14,0.5) circle (0.1pt) node[anchor=west]{};
\filldraw[gray] (0.12,0.5) circle (0.1pt) node[anchor=west]{};
\end{tikzpicture} 
\caption{Example of a Gauss-like map with the choice $I(n)=[\frac{1}{n+1},\frac{1}{n}]$.}
\end{figure}
\end{center}

\subsection{Markov structure and symbolic coding}

We describe now the Markov structure of the maps considered. Given a finite sequence of natural numbers $(a_1,\hdots,a_n)\in\N^n$, the \emph{n-th level cylinder} associated to $(a_1,\hdots,a_n)$ is the set $I(a_1,\hdots,a_n)=I_{a_1}\cap T^{-1}(I(a_2))\cap\hdots\cap T^{-(n-1)}(I(a_n))$. Let $\mathcal{O}=\bigcup_{n}\bigcup_k T^{-n}(\partial I(k))$, then given $x\in [0,1]\setminus \mathcal{O}$ and $n\in\N$, there exists a unique sequence $(a_1(x),a_2(x),\hdots)\in\N^\N$ such that $x\in I(a_1(x),\hdots,a_n(x))$ for every $n$. We denote this sequence  by by $(a_1,a_2,\hdots)$ when $x$ is clear from the context. We also denote $I_n(x)=I(a_1,\hdots,a_n)$ and we say $x$ is coded by the sequence $(a_n)$

Let $\Sigma=\N^\N$ and $\sigma\colon \Sigma\to\Sigma$ be the full shift over $\N$. Then the cylinders in the symbolic space are defined by
\begin{align*}
C(a_1, a_2, \dots ,a_n)=\left\{ (x_n) \in \Sigma \mid x_j=a_j \text{ for } j =1, \dots, n\}\right\}.
\end{align*}

We endow the space $\Sigma$ with the topology generated by the cylinders defined above. Then the map $\pi\colon \Sigma\to I\setminus\mathcal{O}$ given by $\pi((x_n))=\bigcap_n I(x_1,\hdots,x_n)$ is a continuous bijection.

 Given $x\in I\setminus\mathcal{O}$ with coding sequence $(a_n)$ and $n\geq 1$, denote by $I_n^l(x)=I(a_1,\hdots,\newline a_{n-1},≈ a_n+1)$ (resp $I_n^r(x)=I(a_1,\hdots,a_{n-1},a_n-1)$ if $a_n\geq 2$) the level $n$ cylinder on the left (resp right) of $I_n(x)$. Also, denote by $\hat{I}_n(x)=I_n(x)\ \cup \ I_n^r(x)\ \cup \  I_n^l(x)$. If there is no risk of confusion, we omit the dependence on $x$.

Renyi's condition introduced in the previous subsection implies that the length of each cylinder is comparable to the derivative of the iterates of the map at any point of the cylinder. More precisely,
\begin{align*}
0<D^{-1} \leq \left|(T^n)'(x)\right| \cdot |I(a_1,\hdots,a_n)| \leq D
\end{align*}
for every finite sequence $(a_1,\hdots,a_n)\in\N^n$ and $x\in I(a_1,\hdots,a_n)$.

\subsection{The class of measures}\label{conditions}

We start by giving the usual definition of Gibbs measures:

\begin{defi}
Let $\mu$ be an invariant measure with respect to $T$. Then we say that $\mu$ is a \emph{Gibbs measure} associated to the potential $\log \varphi\colon\Sigma \to 
\R$, that is, there exist constants $A,B>0$ so that
\begin{align*}
    A \leq \dfrac{\mu(C(a_1,\hdots,a_n))}{\exp(-nP(\log\varphi)+S_n(\log\varphi)(x))}\leq B,
\end{align*}
where $x$ is any point in $C(a_1,\hdots,a_n)$, $(a_1,\hdots,a_n,\hdots)$ is any sequence in $\Sigma$, $S_n f(x)$ is the Birkhoff sum of $f$ at the point $x$, and $P(\log\varphi)$ is a constant (depending on the potential) called the \emph{topological pressure} of $\log\varphi$.
\end{defi}

Throughout this work we will assume that $P(\log\varphi) = 0$, otherwise we can take the zero pressure potential $\log\varphi-P(\log\varphi) $. It is important to note that it is not trivial that this will not affect our computations, and we will show later how we can overcome that difficulty. The sequence $p_n = \mu(I(n))$ will be of particular relevance for our computations.

We can project this measure to $I$ by setting $\hat{\mu}=\pi^{-1}\circ \mu$. We assume these measures are invariant and ergodic with respect to $T$. We will denote by $\mu$ both the measure in the symbolic space and the projected measure.

Our main assumption on the class of measures is that they have infinite entropy. This can be expressed by saying that the potential $-\log\varphi$ is not integrable with respect to $\mu$. In fact, by the Gibbs property, the Shannon-McMillan-Breiman entropy can be written as
\begin{align*}
    h_\mu = -\lim_{n \to \infty} \dfrac{1}{n}\log\mu \left(C(a_1,\hdots,a_n) \right) = \lim_{n\to\infty} \dfrac{1}{n}S_n(-\log\varphi)(x) = 
    \infty
\end{align*}
for $\mu$ almost every $x\in\Sigma$ if the integral of $-\log\varphi$ is infinite. The last equality is a consequence of Lemma \ref{infinitelyap}.

We define the \emph{n-th variation} of the potential $\log\varphi$ by
\begin{align*}
    \var_n (\log\varphi) = \sup \{ |\log\varphi(x)-\log\varphi(y)| \mid x,y\in I(a_1,\hdots,a_n), (a_1,\hdots,a_n)\in \N^n \}.
\end{align*}

\begin{defi}
Let $x_n$ be the unique fixed point of $T$ in $I(n)$. We define then the \emph{decay ratio} by
\begin{align*}
s=\lim_{n\to\infty}\dfrac{\log \varphi(x_n)}{\log r_n} = \lim_{n\to\infty}\dfrac{\log p_n}{\log r_n}.
\end{align*}
The \emph{tail decay ratio} is defined by
\begin{align*}
\hat{s}=\lim_{n\to\infty} \dfrac{\log\sum_{m\geq n} \varphi(x_m)}{\log\sum_{m\geq n} r_m} = \lim_{n\to\infty} \dfrac{\log\sum_{m\geq n}p_m}{\log\sum_{m\geq n} r_m}.
\end{align*}
\end{defi}

Both definitions for $s$ and $\hat{s}$ agree since $\mu$ is a Gibbs measure. Note also that the definitions above are independent of the choice of the point $x_n$ representing each cylinder if $\var_1(\varphi)<\infty$. By Cers\`aro-Stolz theorem we can write the decay ratio as
\begin{align*}
    s =\lim_{n\to\infty} \dfrac{\sum_{k=1}^n p_n\log p_n}{\sum_{k=1}^n p_n\log r_n}.
\end{align*}

\begin{assu}\label{assumption} Assume that $\var_1(\log\varphi) < \infty$.  
For the sequence sequence $q=\{ q_n\}_{n\in\N} = \{ \varphi(x_n)\}$ we assume that for every $n\in\N$, we have
\begin{align*}
0<K\leq q_{n+1}/q_n \leq K' < \infty 
\end{align*}
for some constants $K,K'$.
\end{assu}

The second condition prevents the existence of large jumps for the potential along sufficiently sparse subsequences of $\{x_n\}$. By the Gibbs property, the properties hold if we replace $q_n$ by $p_n$.

\subsection{Entropy and Lyapunov exponent} 
Since our measures are Gibbs and the potential has finite first variation, we can write the entropy of the system simply as
\begin{align*}
h_\mu=-\sum_{n=1}^\infty q_n\log q_n = -\sum_{n=1}^\infty p_n\log p_n
\end{align*} 
We define the Lyapunov exponent as
\begin{align*}
\lambda_\mu=\int_0^1 \log|T'(x)|\mathrm{d}\mu(x).
\end{align*}
 By the bounded distortion property, we can write the Lyapunov exponent as
\begin{align*}
\lambda_\mu=-\sum_{n=1}^\infty q_n\log r_n+L.
\end{align*}
where $L$ is a distortion constant (independent of $\mu$). Thus, $\lambda_\mu$ is infinite if and only if the series above is divergent. Throughout this work, we assume that both numbers $h_\mu$ and $\lambda_\mu$ are infinite, and hence we can think of $\lambda_\mu$ as defined by the series above.

\subsection{Hausdorff and packing dimension} \label{ssec:haus} In this section we introduce the dimension theory elements we will study throughout this work. Recall the diameter of a set $U\subset \R$ is given by
\begin{align*}
|U| = \sup \{ |x-y| \mid x,y\in U \}.
\end{align*}
For a cover $\mathcal{U}$ of a set $X\subset\R$, its diameter is given by
\begin{align*}
\diam \mathcal{U} = \sup\{| U| \mid U\in\mathcal{U} \}.
\end{align*}

\begin{defi}
Given $X\subset \R$ and $\alpha\in\R$, the $\alpha-$\emph{dimensional Hausdorff measure} \index{Hausdorff measure} of $X$ is given by
\begin{align*}
m(X,\alpha)=\lim_{\delta\to 0}\inf_{\mathcal{U}}\sum_{U\in\mathcal{U}}|U|^\alpha ,
\end{align*}
where the infimum is taken over finite or countable covers $\mathcal{U}$ of $X$ with $\diam\mathcal{U}\leq \delta$.
\end{defi}
It is possible to prove that there exists a number $s\in[0,\infty]$ such that $m(X,\alpha)=\infty$ for $t< s$ and $m(X,\alpha)=0$ for $t> s$, since $m(X,\alpha)$ is decreasing in $\alpha$ for a fixed set $X$.
\begin{defi}
The unique number
\begin{align*}
\dim_H X =\inf\{\alpha\in[0,\infty] \mid m(X,\alpha)=0 \}
\end{align*}
is called the \emph{Hausdorff dimension} \index{Hausdorff dimension} of $X$.
\end{defi}

We extend the notion of Hausdorff dimension to finite Borel measures on $\R$:
\begin{defi}
Let $\mu$ be a finite Borel measure on $\R$. The \emph{Hausdorff dimension} of $\mu$ is defined by
\begin{align*}
\dim_H \mu=\inf\{\dim_H(Z)\mid \mu(\R \setminus Z)=0 \}.
\end{align*}
\end{defi}

We define now the analogue notion of Packing dimension

\begin{defi}
We say that a collection of balls $\{U_n \}_n\subset \R$ is a $\delta-$packing of the set $E\subset\R$ if the diameter of the balls is less than or equal to $\delta$, they are pairwise disjoint and their centres belong to $E$. For $\alpha\in\R$, the $\alpha-$dimensional pre-packing measure of $E$ is given by
\begin{align*}
P(E,\alpha)=\lim_{\delta\to 0}\sup\Big\{ \sum_n \diam (U_n)^\alpha \Big\}
\end{align*}
where the supremum is taken over all $\delta-$packings of $E$. The $\alpha-$dimensional packing measure of $E$ is defined by
\begin{align*}
p(E,\alpha)=\inf\Big\{ \sum_i P(E_i,\alpha) \Big\}
\end{align*}
where the infimum is taken over all covers $\{E_i \}$ of $E$. Finally, we define the packing dimension of $E$ by
\begin{align*}
\dim_P(E)=\sup\{s\mid p(E,\alpha)=\infty \}=\inf\{s\mid p(E,\alpha)=0 \}.
\end{align*}
\end{defi}

We extend the notion of packing dimension to finite Borel measures on $\R$.
\begin{defi}
Let $\mu$ be a finite Borel measure on $\R$. The \emph{Packing dimension} of $\mu$ is defined by
\begin{align*}
\dim_P \mu=\inf\{\dim_P(Z)\mid \mu(\R\setminus Z)=0 \}.
\end{align*}
\end{defi}

It is important to remark that the definitions of dimension for measures is not standard. For instance, a different definition often used is given by
\begin{align*}
    \underline\dim_H \mu&=\inf\{\dim_H(Z)\mid  \mu(Z)>0 \},\\
    \underline\dim_P \mu&=\inf\{\dim_P(Z)\mid \mu(Z)>0 \}.
\end{align*}
We refer to these as lower Hausdorff (packing) dimensions.

Bounding the Hausdorff dimension from above or the Packing dimension from below usually involves the use of a single suitable cover of the space, while for bounds from below and above respectively, we have to deal with \emph{every} cover of the space. There are several tools to help with this problem, and we will make use of the so called (local) Mass Distribution Principles. For this, we introduce the notion of local dimension.

\begin{defi}
The \emph{lower} and \emph{upper pointwise dimensions} \index{Pointwise dimension} of the measure $\mu$ at a point $x\in X$ is given by
\begin{align*}
\underline{d}_\mu(x)=\liminf_{r\to 0}\dfrac{\log \mu(B(x,r))}{\log r} \ , \  \overline{d}_\mu(x)=\limsup_{r\to 0}\dfrac{\log \mu(B(x,r))}{\log r} .
\end{align*}
When both limits coincide, we call the common value the \textit{pointwise dimension} of $\mu$ at $x$ and denote it by $d_\mu(x)$ and say that $\mu$ is \emph{exact dimensional} if $\underline{d}=\overline{d}$ almost everywhere.
\end{defi}

If $d_\mu(x)=d$, then $\mu(B(x,r))\sim r^d$ for small values of $r$. We state now the local version of the Mass Distribution Principle.

\begin{prop} 
Let $X\subset \R$ and $\alpha\in(0,\infty]$, then
\begin{enumerate}
\item If $\underline{d}_\mu(x)\geq \alpha$ for $\mu-$almost every $x\in X$, then $\dim_H \mu \geq \alpha $;
\item If $\underline{d}_\mu(x)\leq \alpha$ for every $x\in X$, then $\dim_H X\leq \alpha$,
\item If $\overline{d}_\mu(x)\geq \alpha$ for $\mu-$almost every $x\in X$, then $\dim_P \mu \geq \alpha $;
\item If $\overline{d}_\mu(x)\leq \alpha$ for every $x\in X$, then $\dim_P X\leq \alpha$,
\item We have
\begin{align*}
\dim_H\mu &=\esssup\{ \underline{d}_\mu(x)\mid x\in X \},\\
\dim_P\mu &=\esssup\{ \overline{d}_\mu(x)\mid x\in X \},\\
\end{align*}
\end{enumerate}
\end{prop}

\begin{proof}
This follows from Proposition 2.3 of \cite{falconer1997techniques}.
\end{proof}

In particular, if $\underline{d}_\mu(x)$ is constant almost everywhere, then $\dim_H \mu$ is equal to that constant value. Analogously, if $\overline{d}_\mu(x)$ is constant almost everywhere, then $\dim_P\mu$ is equal to that constant value.

A notion of dimension which is more adapted to the underlying structure of our dynamical system is the symbolic dimension, which we proceed to define.

\begin{defi}
Given $x\in I$, we define the \emph{lower symbolic dimension} of $\mu$ at $x$ by
\begin{align*}
\underline{\delta}(x) = \liminf_{n\to\infty} \dfrac{\log \mu(I_n(x))}{\log |I_n(x)|},
\end{align*}
and the \emph{upper symbolic dimension} of $\mu$ at $x$ by
\begin{align*}
\overline{\delta}(x) = \limsup_{n\to\infty} \dfrac{\log \mu(I_n(x))}{\log |I_n(x)|},
\end{align*}
If $\overline{\delta}(x)=\underline{\delta}(x)$, then we define the \emph{symbolic dimension} of $\mu$ at $x$ as the common value,  denote it by $\delta(x)$, and we say that $\mu$ is \emph{symbolic exact dimensional} if $\underline{\delta} = \overline{\delta}$.
\end{defi}

\subsection{Main results} The estimates we prove depend on asymptotic relations between the measure and the length of cylinders defining the system. The main results are then:

\begin{teo}
Let $T$ be an EMR map, and $\mu$ be an infinite entropy Gibbs measure satisfying assumption \ref{assumption}. If the decay ratio exists and it is equal to $s$, then
\begin{align*}
    \underline{\delta}(x) = \overline{\delta}(x) = s \ \mu \text{ a.e.}
\end{align*}
\end{teo}

If we assume that the decay of $\{r_n\}$ is polynomial and the measure satisfies the regularity conditions given by assumption \ref{assumption}, we can compute the local dimensions:

\begin{teo}\label{mainteo}
Let $T$ be a Gauss-like map, and $\mu$ be an infinite entropy Gibbs measure satisfying assumption \ref{assumption}. If the decay ratio exists and it is equal to $s$, then 
\begin{enumerate}
    \item $\overline{d}(x) = s \ \mu \text{ a.e.}$,
    \item $\underline{d}(x) = 0 \ \mu \text{ a.e.}$,
\end{enumerate}
Consequently, $0=\dim_H \mu < s = \dim_P \mu$.
\end{teo}

\section{Symbolic dimension} \label{sec:symb}

\subsection{Computation of the symbolic dimension}

We prove now that under the above assumptions, the Gibbs measure $\mu$ is symbolic exact dimensional, and this dimension coincides with the decay ratio. This result does not depend on the length decaying ratio of the partition of the interval.

In general the Lyapunov exponent majorizes the entropy. In a more general setting, this result is known as Ruelle's inequality (see \cite{ruelle1978inequality}).

\begin{prop}
If $h_\mu=\infty$ then $\lambda_\mu=\infty$.
\end{prop}

\begin{proof}
This is an immediate consequence of the Volume Lemma (theorem \ref{volumeLemma}): if $\lambda_\mu\neq \infty$, then $\dim_H \mu = \infty$ which is impossible.
\end{proof}

We prove a well known fact about non-integrable observables.

\begin{lema} \label{infinitelyap}
Let $f\colon [0,1]\to\R$ be a bounded below measurable function such that $\int_0^1 f\mathrm{d}\mu = \infty$. Then
\begin{align*}
\lim_{n\to\infty}\dfrac{1}{n}\sum_{k=0}^{n-1} f \circ T^ k = \infty
\end{align*}
for $\mu$ almost every point.
\end{lema}

\begin{proof}
The proof is an standard application of the Monotone Convergence Theorem. Assume $f$ is positive (otherwise, decompose $f$ into its positive and negative part) and let $M>0$. Then
\begin{align*}
    \liminf_{n\to\infty} \dfrac{1}{n}\sum_{k=0}^{n-1} f\circ T^k(x) &\geq \lim_{n\to\infty} \dfrac{1}{n}\sum_{k=0}^{n-1} \min\{ f\circ T^k,M\}(x)  \\
    &= \int_0^1 \min\{ f,M\}(x) \mathrm{d}\mu (x)
\end{align*}
by Birkhoff's Ergodic Theorem applied to $\min\{ f,M\}$. By the Monotone Convergence Theorem,
\begin{align*}
    \lim_{n\to\infty} \int_0^1 \min\{ f,M\}(x) \mathrm{d}\mu (x) = \int_0^1 f\mathrm{d}\mu(x) = \infty
\end{align*}
from where we conclude the result.
\end{proof}

This result implies in particular that we can assume that the pressure of our potential is zero, as $S_n(\log \varphi)$ dominates $-nP(\log\varphi)$ when $\log\varphi$ is not integrable.

We formulate a lemma regarding the metric and measure theoretic properties of the cylinders associated to the map. This will allow us to write geometric quantities in ergodic theoretic terms. Its proof is a standard applications of the bounded distortion and Gibbs properties.

\begin{lema} \label{metricprop}
For every finite sequence $(a_1,\hdots,a_n)\in\N^n$ and $m\in\N$, we have that
\begin{enumerate}[(a)] 
\item $
 | \log|I(a_1,\hdots,a_n)| - \sum_{k=1}^n\log r_{a_k} | \leq  nD_1+D_2,
$
\item $
\big | \log|\bigcup_{m=0}^t I(a_1,\hdots,a_{n-1},a_n+m)| - \sum_{k=1}^{n-1}\log r_{a_k} - \log\left(\sum_{k=m}^t r_{a_n+k}\right)\big| \leq  nD_1+D_2,
$
\item $
\big | \log|\bigcup_{m=0}^\infty I(a_1,\hdots,a_{n-1},a_n+m)| - \sum_{k=1}^{n-1}\log r_k - \log(\sum_{k=m}^\infty r_{a_n+k})\big| \leq  nD_1+D_2,
$
\item $ | \log \mu (I(a_1,\hdots,a_n)) - \sum_{k=1}^{n}\log p_{a_k} | \leq  nG_1+G_2, $
\item $
 |\log \mu(\bigcup_{m=0}^t I(a_1,\hdots,a_{n-1},a_n+m)) - \sum_{k=1}^{n-1}\log p_{a_k} - \log(\sum_{k=m}^t p_{a_n+k})| \leq  nG_1+G_2,
$
\item $
 |\log \mu(\bigcup_{m=0}^\infty I(a_1,\hdots,a_{n-1},a_n+m)) - \sum_{k=1}^{n-1}\log p_{a_k} - \log\left(\sum_{k=m}^\infty p_{a_n+k}\right)| \leq  nG_1+G_2,
$
\end{enumerate}
where $D_1,D_2$ are distortion constants and $G_1,G_2$ are constants arising from the Gibbs property.
\end{lema}

We proceed to compute the symbolic dimension of our system. This result holds regardless of the decay rate of the sequence $\{r_n\}$.

\begin{teo}\label{markovexact}
Let $T$ be an EMR map and $\mu$ a Gibbs measure with infinte entropy satisfying Assumption 1. Then if the decay ratio exists, we have that $\mu$ is symbolic-exact dimensional and for $\mu$-almost every $x\in I$, 
\begin{align*}
\delta(x)=s.
\end{align*}
\end{teo}

\begin{proof}
By Lemma \ref{infinitelyap} applied to the observables $\log\varphi$ and $\log r_{a_1}$ and Lemma \ref{metricprop}, we have
\begin{align*}
\underline{\delta}(x)\leq  \liminf_{n\to\infty}\dfrac{S_n(\log\varphi)(x)}{- nD_1 - D_2 + S_n(\log r_{a_1})(x)} =  \liminf_{n\to\infty} \dfrac{\log(q_{a_1}\hdots q_{a_n})}{\log(r_{a_1}\hdots r_{a_n})} , \\
\underline{\delta}(x)\geq  \liminf_{n\to\infty}\dfrac{S_n(\log\varphi)(x)}{ nD_1+D_2 + S_n(\log r_{a_1})(x)} =  \liminf_{n\to\infty} \dfrac{\log(q_{a_1}\hdots q_{a_n})}{\log(r_{a_1}\hdots r_{a_n})} ,
\end{align*}
and analogously for the upper symbolic dimension
\begin{align*}
\overline{\delta}(x)= \limsup_{n\to\infty} \dfrac{\log(q_{a_1}\hdots q_{a_n})}{\log(r_{a_1}\hdots r_{a_n})}
\end{align*}
where $(a_1,a_2,\hdots)$ is the sequence coding $x$. With a similar argument, we can also show that
\begin{align*}
    \underline{\delta}(x)=  \liminf_{n\to\infty} \dfrac{\log(p_{a_1}\hdots p_{a_n})}{\log(r_{a_1}\hdots r_{a_n})} ,
\end{align*}
and analogously for the upper symbolic dimension.

For $x\in I$ and $n,k\geq 1$, define 
\begin{align*}
f_{n,k}(x)= \#\{ i\in\{ 1,\hdots, n\}\mid  k_i(x)=k \},
\end{align*}
that is, the number of times the orbit of $x$ visits the interval $I_k$ in the first $n$ steps. Recall that from the Birkhoff Theorem, we have that 
\begin{align*}
\lim_{n\to\infty}\dfrac{f_{n,k}}{n} = p_k
\end{align*}
for $\mu-$almost every $x\in I$. In particular, the orbit of almost every $x\in I$ visits every cylinder $I(n)$ infinitely many times. Fix $x$ in the set where the convergence holds, and then define $m\colon \N\to \N $ by $m(n)=\max\{ k_i(x)\mid i\in\{1,\hdots,n \} \}$. The previous remark shows that $m$ is unbounded, and it is clearly non-decreasing. Thus, we can write
\begin{align*}
-\log(r_{k_1}\hdots r_{k_n})=-\sum_{j=1}^n\log r_{k_j} = -\sum_{j=1}^{m(n)}f_{n,j}\log r_{j}.
\end{align*}

Given $\epsilon>0$, there exists $n_1$ such that
\begin{align*}
\Bigg| \dfrac{\log p_k}{\log r_k} -s \Bigg|<\epsilon
\end{align*}
for every $k\geq n_1$, that is, $(-\log p_k) < (\epsilon+s)(-\log r_k)$ for $k\geq n_1$. For $n$ large enough so that $m(n)>n_1$, we write
\begin{align*}
\dfrac{\log(p_{k_1}\hdots p_{k_n})}{\log(r_{k_1}\hdots r_{k_n})}=\dfrac{\displaystyle\sum_{k=1}^{n_1}f_{n,k}(-\log p_k)+\sum_{k=n_1+1}^{m(n)}f_{n,k}(-\log p_k)}{\displaystyle\sum_{k=1}^{n_1}f_{n,k}(-\log r_k)+\sum_{k=n_1+1}^{m(n)}f_{n,k}(-\log r_k)}.
\end{align*}
We analyse separately the two bits of the sum: 
\begin{align*}
S_1(n)&=\dfrac{\displaystyle\sum_{k=1}^{n_1}f_{n,k}(-\log p_k)}{\displaystyle\sum_{k=1}^{n_1}f_{n,k}(-\log r_k)+\sum_{k=n_1+1}^{m(n)}f_{n,k}(-\log r_k)},\\
S_2(n)&=\dfrac{\displaystyle\sum_{k=n_1+1}^{m(n)}f_{n,k}(-\log p_k)}{\displaystyle\sum_{k=1}^{n_1}f_{n,k}(-\log r_k)+\sum_{k=n_1+1}^{m(n)}f_{n,k}(-\log r_k)}.
\end{align*}
For $k=1,\hdots, n_1$ taking $\epsilon_k=p_k/2$ there exists $n_2\geq n_1$ such that
\begin{align*}
\dfrac{np_k}{2}\leq f_{n,k} \leq \dfrac{3np_k}{2}
\end{align*}
for every $n\geq n_2$. Thus, the terms $\sum_{k=1}^{n_1}f_{n,k}(-\log p_k)$ and $\sum_{k=1}^{n_1}f_{n,k}(-\log r_k)$ grow linearly in $n$ for $n$ large enough. We will show that $\sum_{k=n_1+1}^{m(n)}f_{n,k}(-\log r_k)$ grows faster than linear. 

Given $M>0$, since the Lyapunov exponent is infinite, there exists $n_3$ such that
\begin{align*}
\sum_{k=n_1+1}^{m(n)}p_k(-\log r_k) > 2M
\end{align*}
for every $n\geq n_3$. Now, for $k=n_1+1,\hdots,m(n_3)$, take $\epsilon_k=p_k/2$ and so there exists $n_4\geq n_3$ such that
\begin{align*}
f_{n,k}\geq \dfrac{np_k}{2}
\end{align*}
for every $n\geq n_4$ and $k=n_1+1,\hdots,m(n_3)$. Thus
\begin{align*}
\dfrac{1}{n}\sum_{k=n_1+1}^{m(n)} f_{n,k}(-\log r_k)&=\dfrac{1}{n}\sum_{k=n_1+1}^{m(n_4)} f_{n,k}(-\log r_k)+\dfrac{1}{n}\sum_{k=m(n_4)+1}^{m(n)} f_{n,k}(-\log r_k)\\
&\geq \dfrac{1}{n} \sum_{k=n_1+1}^{m(n_4)} \dfrac{np_k}{2}(-\log r_k) \\
&=\dfrac{1}{2}\sum_{k=n_1+1}^{m(n_4)} p_k(-\log r_k) > M
\end{align*}
for every $n\geq n_4$. This shows that $S_1(n)\to 0$ as $n\to\infty$. To estimate $S_2(n)$, we note that
\begin{align*}
S_2(n) \leq (s+\epsilon)\cdot \dfrac{\displaystyle\sum_{k=n_1+1}^{m(n)}f_{n,k}(-\log r_k)}{\displaystyle\sum_{k=1}^{n_1}f_{n,k}(-\log r_k)+\sum_{k=n_1+1}^{m(n)}f_{n,k}(-\log r_k)}
\end{align*}
Using the same argument as above, we can show that $\sum_{k=n_1+1}^{m(n)}f_{n,k}(-\log r_k)$ grows faster than linear, so $\lim S_2(n)\leq s+\epsilon$. This shows that 
\begin{align*}
\delta(x) \leq s
\end{align*}
The proof of the opposite inequality is analogous.
\end{proof}

\subsection{The decay ratio}

Now we proceed to study the properties of the decay. In fact, we show that for infinite entropy measures, it is completely determined by the properties of the partition $\{I(n)\mid n\in\N \}$:
\begin{defi}
The \emph{convergence exponent} of the partition $\{r_n \}$ of $I$ is defined by
\begin{align*}
    s_\infty = \inf \Big\{ s\geq 0 \mid \sum_{n=1}^\infty r_n^s < \infty \Big\}.
\end{align*}
\end{defi}

\begin{prop}\label{convexp}
In general, we have that $s_\infty \leq s$. Under the assumption that $h_\mu=\infty$, we also have $s\leq s_\infty$.
\end{prop}

\begin{proof}
Given $\epsilon>0$, there exists $n_1$ such that 
\begin{align*}
    (\epsilon+s)\log r_n < \log p_n < (s-\epsilon)\log r_n
\end{align*}
for every $n\geq n_1$, and thus $r_n^{s+\epsilon}<p_n $ for every $n\geq n_1$. Summing over $n$ we get
\begin{align*}
    \sum_{n=1}^\infty r_n^{s+\epsilon}= \sum_{n=1}^{n_1-1} r_n^{s+\epsilon} + \sum_{n=n_1}^\infty r_n^{s+\epsilon} \leq \sum_{n=1}^{n_1-1} r_n^{s+\epsilon}+\sum_{n=n_1}^\infty p_n <\infty.
\end{align*}
Hence, $s_\infty\leq s+\epsilon$ for every $\epsilon>0$ and so $s_\infty\leq s$.

Now, Suppose that $s_\infty < s$, and hence, there is $\alpha > 0$ such that $s_\infty \leq s_\infty + \alpha  < s$ and
\begin{align*}
    \sum_{n=1}^\infty r_n^{s_\infty + \alpha} < \infty.
\end{align*}
Let $\epsilon = (s-s_\infty - \alpha)/2>0$, then there is an integer $n_0$ such that
\begin{align*}
    r_n^{s+\epsilon} \leq p_n \leq r_n^{s-\epsilon}
\end{align*}
for all $n\geq n_0$. This implies that
\begin{align*}
    \sum_{n=n_0}^\infty p_n(-\log p_n) \leq (s+\epsilon) \sum_{n=n_0}^\infty r_n^{s-\epsilon}(-\log r_n).
\end{align*}

Recall the one sided limit criterion for convergence of series: let $a_b,b_n > 0$ sequences such that
\begin{align*}
    \limsup_{n\to\infty} \dfrac{a_n}{b_n} = c \in \left[ 0,\infty\right)
\end{align*}
and $\sum b_n < \infty$. Then $\sum a_n <\infty$.

Let $f\colon [0,\infty)\to \R$ be the function defined by
\begin{align*}
    f(x) = \left\{\begin{array}{lr}
        0, & \text{for } x =0,\\
        x^\epsilon(-\log x), & \text{for } x>0.
        \end{array} \right.
\end{align*}
It is easy to see that $f$ is continuous. Taking $a_n=r_n^{s-\epsilon}(-\log r_n)$ and $b_n=r_n^{s_\infty+\alpha}$ and using the continuity of $f$, we get that
\begin{align*}
    \limsup_{n\to\infty} \dfrac{a_n}{b_n} = \lim_{n\to\infty} r_n^\epsilon (-\log r_n) = 0.
\end{align*}
We conclude that 
\begin{align*}
\sum_{n=n_0}^\infty p_n(-\log p_n) \leq (s+\epsilon) \sum_{n=n_0}^\infty r_n^{s-\epsilon}(-\log r_n) < \infty,
\end{align*}
contradicting the fact that the entropy is infinite.

\end{proof}

We give now a definition for the asymptotic decay of the sequence $\{ r_n\}$.

\begin{defi}\label{decay}
The \emph{asymptotic} of the partition $\{r_n \}$ is defined as
\begin{align*}
     \alpha = \sup\{ t\geq 0 \mid  \lim_{n\to\infty} n^t r_n < \infty\}.
\end{align*}
We say that $\{ r_n\}$ decays \emph{polynomially} if $\alpha > 1$, and we say that $\{r_n\}$ decays \emph{superpolynomially} if $\alpha=\infty$.
\end{defi}

Note that if $r_n$ has polynomial decay with asymptotic $\alpha$, then $s_\infty = 1/\alpha$. If we know the asymptotic of $\{r_n\}$, we can compute the asymptotic of the tail of the series of $\{r_n\}$:

\begin{lema}\label{taildecay}
If the asymptotic of $\{r_n\}$ is $\alpha > 1$, then the asymptotic of $R_n = \sum_{m\geq n} r_n$ is $\alpha - 1$.
\end{lema}

\begin{proof}
It suffices to show that the sets $A=\{ t\geq 1 \mid  \lim_{n\to\infty} n^t r_n < \infty\}$ and $A'=\{ t\geq 0 \mid  \lim_{n\to\infty} n^{t-1} R_n < \infty\}$ are the same. Let $t\in A$, then $\lim_{n\to\infty} n^t r_n = d$, and so given $\epsilon$, there is $n_0\in\N$ such that
\begin{align*}
    \dfrac{(d-\epsilon)}{n^t} < r_n < \dfrac{(d+\epsilon)}{n^t}.
\end{align*}
for $n\geq n_0$. Hence, for $n\geq n_0$,
\begin{align*}
 \dfrac{(d-\epsilon)}{(t-1)}\leq \sum_{m=n}^\infty \dfrac{n^{t-1}(d-\epsilon)}{m^t} 
    \leq n^{t-1}R_n \leq \sum_{m=n}^\infty \dfrac{n^{t-1}(d+\epsilon)}{m^t} \leq \dfrac{n^{t-1}(d+\epsilon)}{(n+1)^{t-1}(t-1)}
\end{align*}
from which follows that $t-1\in A'$. Now, if $t\in A'$, we have that $\lim_{n\to\infty} n^{t-1} R_n = d' < \infty$, and thus, given $\epsilon>0$, there is $n_1\in\N$ such that
\begin{align*}
    \dfrac{-\epsilon+d'}{n^t}\leq \sum_{m\geq n} r_n \leq \dfrac{\epsilon+d'}{n^t}.
\end{align*}
This implies that
\begin{align*}
     \dfrac{(-\epsilon+d')}{n^t}- \dfrac{(\epsilon+d')}{(n+1)^t} \leq  r_n \leq \dfrac{(\epsilon+d')}{n^t}-\dfrac{(-\epsilon+d')}{(n+1)^t}.
\end{align*}
from which follows that $t+1\in A$, proving the assertion.
\end{proof}

\section{Infinite ergodic theory} \label{sec:infinite}

In this section we explore the consequences of the non-integrability of the function $-\log r_{a_1}$ (or equivalently, $\lambda_\mu = \infty$). Using tools of infinite ergodic theory we can prove that the diameter of the cylinders decreases faster than exponentially from a given level to the next.

\subsection{Finite Lyapunov exponent argument}
We proceed to show now one of the usual arguments used to compute Hausdorff dimensions and remark how it fails in our case.

\begin{lema}
Let $T$ be an EMR map and $\mu$ a Gibbs measure. Then for almost every $x\in I$ and every $r>0$ there exists $n$ such that 
\begin{align}
\dfrac{\log\mu(B(x,r))}{\log r} \leq \dfrac{\log \mu(I_{n}(x))}{\log|I_{n-1}(x)|}.  \label{ineq}
\end{align}
\end{lema}

\begin{proof}
This is a well known argument and can be found for instance in \cite{pesin2008dimension}. Given $r>0$, there exists a unique integer $n=n(r)$ such that 
\begin{align*}
|I_{n}(x)|<r \leq |I_{n-1}(x)|
\end{align*}
so then
\begin{align*}
I_n(x)\subseteq B(x,|I_{n}(x)|)\subseteq B(x,r) \subseteq B(x,|I_{n-1}(x)|).
\end{align*}
Then
\begin{align*}
\log \mu (I_n(x))\leq \log \mu (B(x,r)),
\end{align*}
and since $\log r \leq \log |I_{n-1}(x)|$, we obtain
\begin{align*}
\dfrac{\log \mu(B(x,r))}{\log r} \leq \dfrac{\log \mu (I_n(x))}{\log |I_{n-1}(x)|} 
\end{align*}
as we wanted.
\end{proof}

In a similar way, it is possible to show that
\begin{align}
 \dfrac{\log C_1\mu(I_{n-1}(x))}{\log C_2|I_{n}(x)|} \leq \dfrac{\log \mu(B(x,r))}{\log r} \label{ineq1}
\end{align}
where $C_1,C_2$ are constants arising from assumption \ref{assumption} and Renyi's property respectively. Note that if $\lambda_\mu<\infty$, then inequalities (\ref{ineq}) and (\ref{ineq1})
, and the Ergodic Theorem would immediately imply that $s =\dim_H \mu =\dim_P \mu $. However, since in our case $\lambda_\mu=\infty$, the previous argument does not work. In fact, here lies the main difficulty of the infinite entropy and Lyapunov exponent case. The following lemma shows that the situation is as bad as it can get: for almost every point, the diameter of the cylinders decreases arbitrarily from one level to the next. 


\begin{prop}\label{limsup}
Let $T$ be a Gauss-like map and $\mu$ an infinite entropy Gibbs satisfying assumption \ref{assumption}. Then for almost every $x\in I$, we have that
\begin{align*}
\liminf_{n\to\infty} \dfrac{\log |I_n(x)|}{\log |I_{n-1(x)}|}=1,
\end{align*}
and
\begin{align*}
\limsup_{n\to\infty} \dfrac{\log |I_n(x)|}{\log |I_{n-1}(x)|}=\infty.
\end{align*}
\end{prop}

The proof of the first equality is an immediate consequence of recurrence. We postpone the proof of the second equality. We will return to this issue once we set up the appropriate tools to prove this result.

\begin{coro}
For almost every $x\in I$, we have that $\underline{d}(x)\leq s$ and hence $\dim_H \mu\leq s$.
\end{coro}

The main tool that we will use to prove proposition \ref{limsup} are results about the pointwise behavior of trimmed sums.

\subsection{Trimmed convergence}
Note that the sequence $\{X_n=-\log r_1\circ T^{n-1} \} $ can be seen as a positive ergodic stationary process on $[0,1]$ with respect to $\mu$, an infinite entropy Gibbs measure satisfying assumption \ref{assumption}. The distribution function of $X_1$ is $\F(t)=\mu(X_1\geq t)$, and it can be seen that $\E(X_1)=\lambda_\mu$.  As we saw in Lemma \ref{infinitelyap}, the Ergodic Theorem fails to provide non-trivial information. This result was vastly generalized by Robbins and Chow for i.i.d. random variables in \cite{chow1961sums} and in the ergodic stationary case by Aaronson in \cite{aaronson1977ergodic} who proved the following theorem:
\begin{teo}\cite{aaronson1977ergodic}\label{aar}
Let $f\colon [0,1]\to\R$ be a non-negative measurable function. If $\E(f)=\infty$ then for any sequence $\{b_n\}$ of positive numbers, either
\begin{align*}
\limsup_{n\to\infty} \dfrac{1}{b_n}\sum_{k=0}^{n-1} f\circ T^k =\infty \quad \text{a.e.}
\end{align*}
or
\begin{align*}
\lim_{n\to\infty}\dfrac{1}{b_n}\sum_{k=0}^{n-1} f\circ T^k=0 \quad \text{a.e.}.
\end{align*}
\end{teo}
It is possible to prove that the lack of convergence in the previous theorem is due to a finite number of terms which are not comparable in size to the rest of the terms of the sum. This was proved in the i.i.d. case by Mori in \cite{mori1976strong},\cite{mori1977stability} and in the stationary ergodic case by Aaronson and Nakada in \cite{aaronson2003trimmed}. We formulate the result by Aaronson and Nakada in a setting appropriate for our purposes.

We denote the ergodic sum of a function $f$ by $S_n(f)$ and define $S'_n(f)=S_n(f)-\max\{f,\hdots,f\circ T^{n-1}\}$. We refer to $S'_n$ as the \emph{trimmed ergodic sum} of $f$.

\begin{teo}\cite{aaronson2003trimmed} \label{aarnak}
Let $(X_1,X_2,\hdots)$ be a non-negative, ergodic stationary process with $L(t)=\E(X\wedge t)$, and set $\varepsilon(t):=t(\log L)'(t)$. Suppose that the process is \emph{continued fraction mixing} with exponential rate (see \cite{aaronson2003trimmed}), and that
\begin{align*}
    \sum_{n=1}^\infty \dfrac{\varepsilon^2(n)}{n}<\infty.
\end{align*}
Then, there exists a sequence $\{b_n\}$ such that
\begin{align*}
    \lim_{n\to\infty} \dfrac{S'_n}{b_n} = 1
\end{align*}
almost surely, and we say that $\{ X_n\}$ has trimmed convergence. 
\end{teo}

As remarked in \cite{aaronson2003trimmed}, any Gibbs-Markov map is CF-mixing with exponential rate. For our particular random variables, the series in the previous theorem can be explicitly expressed in terms of the sequences $\{ p_n\}$ and $\{ r_n\}$:
\begin{lema} \label{trimsums}
Suppose that
\begin{align*}
\sum_{n=1}^\infty (\log r_n)^2 (p_n^2+2p_np_{n+1}) <\infty.
\end{align*}
Then the sequence $\{X_n=-\log r_1\circ T^{n-1} \}$ has trimmed convergence.
\end{lema}

\begin{proof}
We show that if 
\begin{align*}
\sum_{n=1}^\infty (\log r_n)^2 (p_n^2+2p_np_{n+1}) <\infty,
\end{align*}
then 
\begin{align*}
    \sum_{n=1}^\infty \dfrac{\varepsilon^2(n)}{n}<\infty.
\end{align*}
Let $\F(t)=\mu(X\geq t)$ and note that 
\begin{align*}
    (\log L)'(t) = \dfrac{\F(t)}{L^2(t)},
\end{align*}
and so
\begin{align*}
    \sum_{n=1}^\infty \dfrac{\varepsilon^2(n)}{n} =  \sum_{n=1}^\infty \dfrac{n\F^2(n)}{L^2(n)} \leq  \sum_{n=1}^\infty n\F^2(n).
\end{align*}
We compare the above sum to the corresponding integral. We can then see that if $x\in[0,-\log r_1)$ then $\F(x)=1$, while if $x\in[-\log r_n,-\log r_{n+1})$ for $n\geq 1$ then
\begin{align*}
\F(x)=\sum_{k=n+1}^\infty p_k,
\end{align*}
so then the integral is
\begin{align*}
\int_0^\infty x\ (\F(x))^2  \mathrm{d}x &= \int_0^{-\log r_1} x \left( \sum_{k=1}^\infty p_k \right)^2 \mathrm{d}x  +  \sum_{n=1}^\infty\left(    \int_{-\log r_n}^{-\log r_{n+1}} x\left( \sum_{k=n}^\infty  p_k   \right)^2  \mathrm{d}x \right) \\
&= \dfrac{(\log r_1)^2}{2} + \dfrac{1}{2} \sum_{n=1}^\infty\left(    \int_{-\log r_n}^{-\log r_{n+1}} x\left(\sum_{i,j=n}^\infty  p_ip_j  \right)  \mathrm{d}x \right) \\
&= \dfrac{(\log r_1)^2}{2} + \dfrac{1}{2}\sum_{n=1}^\infty\left( (\log r_{n-1})^2-(\log r_n)^2 \right) \left( \sum_{i,j=n}^\infty  p_ip_j  \right).
\end{align*}
Call now
\begin{align*}
a_n = (\log r_n)^2 \quad , \quad b_n = \sum_{i,j=n}^\infty  p_ip_j .
\end{align*}
Then, the above expression has the form
\begin{align*}
\sum_{n=1}^\infty (a_{n+1}-a_n)b_n
\end{align*}
which can be written as
\begin{align*}
-a_1b_1 +\sum_{n=1}^\infty a_{n+1}(b_{n}-b_{n+1}).
\end{align*}
Note that
\begin{align*}
b_{n+1}-b_n &= 2p_np_{n+1}+p_n^2 \\
b_1&=1.
\end{align*}
With this, the integral becomes 
\begin{align*}
\int_0^\infty x\ (\F(x))^2  \mathrm{d}x &=  \dfrac{(\log r_1)^2}{2} + \dfrac{1}{2}\left( -(\log r_1)^2 + \sum_{n=1}^\infty (\log r_n)^2 (p_n^2+2p_np_{n+1}) \right) \\
&=\sum_{n=1}^\infty (\log r_n)^2 (p_n^2+2p_np_{n+1}) 
\end{align*}
as we wanted to prove.
\end{proof}

We show now that the trimmed convergence condition is satisfied by systems for which $\{r_n\}$ decays polynomially or slower.
\begin{lema}
Suppose that 
\begin{align*}
	\lim_{n\to\infty}	\dfrac{1}{n}(\log r_n)^2 = c\in [0,\infty).
\end{align*}
Then the sequence $\{X_n=-\log r_n \}$ has trimmed convergence.
\end{lema}
\begin{proof}
Since $p_n$ and $p_{n+1}$ are comparable, it suffices to prove that
\begin{align*}
    \sum_{n=1}^\infty (\log r_n)^2 p_n^2 <\infty.
\end{align*}
Note that $\{ p_n\}\subset \ell^2$ and we have that
\begin{align*}
1=\left( \sum_{n=1}^\infty p_n \right)^2 = \sum_{i,j=1}^\infty p_ip_j.
\end{align*}
Since the sequence $\{ p_n\}$ is decreasing, we have that
\begin{align*}
    \sum_{j=2}^\infty p_j^2(j-1) = \sum_{j=2}^\infty p_j\sum_{i=1}^{j-1}p_j \leq \sum_{j=2}^\infty p_j\sum_{i=1}^{j-1}p_i \leq \sum_{j=2}^\infty p_j\sum_{i=1}^\infty p_i      < \infty.
\end{align*}
Comparing in the limit the series of the left hand side to the series $\sum_n p_n^2 (\log r_n)^2$, we get that this series converge. 
\end{proof}

\begin{coro}
If $T$ is a Gauss-like map, then it has trimmed convergence.
\end{coro}

Now we are in position to prove Lemma \ref{limsup}:

\begin{proof}[Proof of Lemma \ref{limsup}]
Let $x$ be a point with coding sequence $(a_n)$. With an argument analogue to the one used in the proof of Theorem \ref{markovexact}, the limit in question is equivalent to 
\begin{align*}
    \limsup_{n\to\infty} \dfrac{\sum_{k=1}^n \log r_{a_k} }{\sum_{k=1}^{n-1}\log r_{a_k}} = 1+ \limsup_{n\to\infty}  \dfrac{\log r_{a_n}}{\log(r_{a_1}\hdots r_{k_{n-1}})} = 1+ \limsup_{n\to\infty} \dfrac{X_n(x)}{S_{n-1}(x)}
\end{align*}
By Lemma \ref{aarnak}, there exists a sequence $\{b_n \}$ and a set $Z_1\subseteq I$ of full measure such that
\begin{align*}
\lim_{n\to\infty}\dfrac{S'_n}{b_n}  = 1  \text{\quad for a.e. } x\in Z_1.
\end{align*}
Now by Lemma \ref{aar}, there exists a subset of full measure of $I$ such that
\begin{align*}
\limsup_{n\to\infty}\dfrac{S_n}{b_n}=\infty \quad \text{a.e.}
\end{align*}
or
\begin{align*}
\liminf_{n\to\infty}\dfrac{S_n}{b_n}=0 \quad \text{a.e.}.
\end{align*}
Since the trimmed sum is $o(b_n)$, the first condition must hold in a set of full measure $Z_2$. Let $Z=Z_1\cap Z_2$ and $x\in Z$. Given $1>\varepsilon >0$, there exists $n_0$ such that
\begin{align*}
\Big| \dfrac{S'_n}{b_n} - 1\Big| < \varepsilon 
\end{align*}
for every $n\geq n_0$ at $x$. Since $\limsup\frac{S_n}{b_n}=\infty$, given an integer $M>0$ there exists $n_1\geq n_0$ such that 
\begin{align*}
\dfrac{S_{n_1}}{b_{n_1}} > 2M+1
\end{align*}
at $x$. Combining these two inequalities, we obtain
\begin{align*}
\Big| \dfrac{\max\{X_1\hdots, X_{n_1} \}}{b_{n_1}} \Big| = \Big| \dfrac{S_{n_1}}{b_{n_1}} -  \dfrac{S'_{n_1}}{b_{n_1}} \Big| > 2M.
\end{align*}
Now, there exists an index $j\in\{1,\hdots,n_1 \}$ such that $X_j= \max\{X_1\hdots, X_{n_1}\}$ at $x$, and so $S'_j=S_{j-1}$. Since the $X_i$ are positive, we have that 
\begin{align*}
S_{j-1}= \ S'_j\leq \  S'_{n_1} < b_{n_1}(1+\varepsilon)<2b_{n_1}<\dfrac{\max\{X_1\hdots, X_{n_1}\}}{M} = \dfrac{X_j}{M}, 
\end{align*}
and hence 
\begin{align*}
M<\dfrac{X_j}{S_{j-1}}.
\end{align*}
This implies that
\begin{align*}
\limsup_{n\to\infty} \dfrac{X_n}{S_{n-1}}=\infty
\end{align*}
and so 
\begin{align*}
\limsup_{n\to\infty} \dfrac{\log |I_n|}{\log |I_{n-1}|}=\infty
\end{align*}
as we wanted to prove.
\end{proof}

\section{Computing the Hausdorff dimension} \label{sec:haus}

With the tools developed in the previous sections, we proceed with the dimension computations.

Now we prove an upper bound for $\dim_H \mu$. This bound is related to the tail decay ratio $\hat{s}$. We prove two necessary lemmas to give the desired bound. The first lemma shows that $\{p_n\}$ decays slower than any polynomial, while the second lemma, shows the existence of $\hat{s}$ and that $\hat{s}=0$ for Gauss-like maps.

\begin{lema}
Suppose that the decay ratio exists and it is equal to $s$, the sequence $\{r_n\}$ decays polynomially and the measure $\mu$ has infinite entropy. Then for all $\delta > 0$, there exist constants $C,n_0$ such that
\begin{align*}
    p_n \geq \dfrac{C}{n^{1+\delta}}
\end{align*}
for all $n\geq n_0$.
\end{lema}

\begin{proof}
Let $\alpha > 0$ be the polynomial decay of $r_n$. Then by proposition \ref{convexp}, $s = s_\infty = 1/\alpha$, we can take $\epsilon>0$ small enough so that $\epsilon\alpha+\epsilon s + \epsilon^2 < \delta$. Then there exists $C>0$ and $n_0\in\N$ such that
\begin{align*}
    \dfrac{C}{n^{\alpha+\epsilon}}&\leq r_n \\
    \log r_n^{s+\epsilon} &\leq \log p_n
\end{align*}
for all $n\geq n_0$. This implies that 
\begin{align*}
    \dfrac{C^{s+\epsilon}}{n^{1+\delta}}\leq \dfrac{C^{s+\epsilon}}{n^{(\alpha+\epsilon)(s+\epsilon)}} \leq p_n
\end{align*}
for all $n\geq n_0$ as we wanted.
\end{proof}

\begin{lema}
 Under the same assumptions of the previous lemma, the tails decay ratio $\hat s$ exists and is equal to zero.
\end{lema}

\begin{proof}
By the lemma above, for $\delta>0$, there are constants $C,n_0$ such that
\begin{align*}
    p_n \geq \dfrac{C}{n^{1+\delta}}
\end{align*}
for all $n\geq n_0$. This implies that 
\begin{align*}
    \sum_{m= n}^\infty p_m \geq \dfrac{C}{\delta n^\delta}
\end{align*}
for $n\geq n_0$. On the other hand, if we take $\epsilon < \alpha - 1$, there exists $n_1$ such that
\begin{align*}
    r_n \leq \dfrac{C}{n^{\alpha-\epsilon}}
\end{align*}
for $n\geq n_1$ and consequently,
\begin{align*}
 \sum_{m= n}^\infty r_m \leq \dfrac{C}{(\alpha-\epsilon-1)n^{\alpha-\epsilon-1}} 
\end{align*}
 for $n\geq n_1$. Hence
 \begin{align*}
     \dfrac{\log  \sum_{m= n}^\infty p_m }{\log  \sum_{m= n}^\infty r_m} \leq \dfrac{\log C-\log\delta -\delta\log n}{\log C -\log(\alpha-\epsilon-1)-(\alpha-\epsilon-1)\log n}
 \end{align*}
 for $n\geq \max\{n_0,n_1 \}$. This implies that
 \begin{align*}
       \limsup_{n\to\infty} \dfrac{\log  \sum_{m= n}^\infty p_m }{\log  \sum_{m= n}^\infty r_m} \leq \dfrac{\delta}{(\alpha-\epsilon-1)}.
 \end{align*}
 Letting $\delta\to 0$ we conclude the result.
\end{proof}

Now we can compute the lower local dimension, and consequently, obtain the Hausdorff dimension of the measure.

\begin{prop}\label{upperbound}
Suppose $T$ is a Gauss-like map and $\mu$ is an infinite entropy Gibbs measure satisfying assumption \ref{assumption}. Then 
\begin{align*}
\liminf_{r\to 0} \dfrac{\log \mu(B(x,r))}{\log r} \leq \hat{s}
\end{align*}
for $\mu$ almost every $x\in I$.
\end{prop}

\begin{proof}
Let $x$ be a point where Theorems \eqref{markovexact} and \eqref{limsup}, and Lemma \ref{aar} hold (such set is of full measure). Given such $x$ and $n\in\N$, take 
\begin{align*}
r_n=\bigg| \bigcup_{m=0}^\infty I_n^{m\cdot\ell}(x) \bigg|,
\end{align*}
where $I_n^{m\cdot\ell}(x)=I(a_1(x),\hdots, a_{n-1}(x),a_n(x)+m)$. Then
\begin{align*}
\bigcup_{m=0}^\infty I_n^{m\cdot\ell}(x) \subseteq B(x,r_n)
\end{align*}
and so
\begin{align*}
\dfrac{\log\mu(B(x,r_n))}{\log r_n} \leq \dfrac{\log\mu\left(\bigcup_{m=0}^\infty I_n^{m\cdot\ell}(x)  \right)}{\log\bigg| \bigcup_{m=0}^\infty I_n^{m\cdot\ell}(x) \bigg|}.
\end{align*}
Note now that the above inequality can be expressed in terms of the sequences $\{p_n\},\{r_n\}$ using Lemma \ref{metricprop}
\begin{align*}
\log\mu\left(\bigcup_{m=0}^\infty I_n^{m\cdot\ell}(x)  \right)& \geq \sum_{k=1}^{n-1}\log p_{a_k} + \log\left( \sum_{m=0}^\infty p_{a_n+m} \right)	-nG_1 - G_2 \\
\log\bigg| \bigcup_{m=0}^\infty I_n^{m\cdot\ell}(x) \bigg| &\leq \sum_{k=1}^{n-1}\log r_{a_k} + \log\left( \sum_{m=0}^\infty r_{a_n+m} \right)+nD_1  + D_2
\end{align*}
where $G_1,G_2$ are constants arising from the Gibbs property and the finite first variation of the potential, and $D_1,D_2$ are constants arising from the bounded distortion property. Thus, we have
\begin{align*}
\dfrac{\log\mu(B(x,r_n))}{\log r_n} \leq \dfrac{\sum_{k=1}^{n-1}\log p_{a_k} + \log\left( \sum_{m=0}^\infty p_{a_n+m} \right)	-nG_1 - G_2}{\sum_{k=1}^{n-1}\log r_{a_k} + \log\left( \sum_{m=0}^\infty r_{a_n+m} \right)+nD_1  + D_2}.
\end{align*}
For $n$ large enough, we have that 
\begin{align*}
-\epsilon+\hat{s}< \dfrac{-\log\sum_{m=0}^\infty p_{n+m}}{-\log\sum_{m=0}^\infty r_{n+m}} < \hat{s}+\epsilon
\end{align*}
and 
\begin{align*}
-\epsilon+s< \dfrac{-\sum_{k=1}^{n-1}\log p_{a_k}}{-\sum_{k=1}^{n-1}\log r_{a_k}} < s+\epsilon.
\end{align*}
Thus, if $a_n$ is large enough, we have 
\begin{align*}
\dfrac{\log\mu(B(x,r_n))}{\log r_n} \leq  \dfrac{(s + \epsilon)\left(\sum_{k=1}^{n-1}\log r_{a_k}\right) + (\hat{s}+\epsilon)\log\left( \sum_{m=0}^\infty r_{a_n+m} \right)-nG_1-G_2}{\sum_{k=1}^{n-1}\log r_{a_k} + \log\left( \sum_{m=0}^\infty r_{a_n+m} \right)+nD_1+D_2}.
\end{align*}
If $\alpha>1$ is the polynomial decaying ratio of $\{r_n\}$, then by Lemma \ref{taildecay} we get the tail decay asymptotic
\begin{align*}
\sum_{m=0}^\infty r_{n+m}\asymp \dfrac{1}{n^{\alpha-1}}.
\end{align*}
We can then rewrite the above inequality as
\begin{align*}
    \dfrac{\log\mu(B(x,r_n))}{\log r_n} \leq  \dfrac{(s + \epsilon)\left(\sum_{k=1}^{n-1}\log r_{a_k}\right) + (\hat{s}+\epsilon)K(\alpha-1)\log\left( r_{a_n} \right)-nG_1-G_2}{\sum_{k=1}^{n-1}\log r_{a_k} + K(\alpha-1)\log\left( r_{a_n} \right)+nD_1+D_2}.
\end{align*}
where $K$ is the constant implied in the tail asymptotic for $\{r_n\}$.
 By Lemma \ref{aar} and Proposition \ref{limsup}, we can take an increasing subsequence $a_{n_k}$ so that
\begin{align*}
\lim_{k\to\infty}\dfrac{-\log r_{a_{n_k}}}{-\sum_{k=1}^{n_k-1}\log r_{a_k}}&=\infty, \\
\lim_{k\to\infty}-\dfrac{1}{n_k}\log r_{a_{n_k}}&=\infty.
\end{align*}
We get then
\begin{align*}
\lim_{k\to\infty}\dfrac{\log\mu(B(x,r_{n_k}))}{\log r_{n_k}} \leq  \hat{s}+\epsilon
\end{align*}
Letting $\epsilon\to 0$ we conclude that $\underline{d}(x)\leq \hat{s}$ as we wanted.
\end{proof}

From the above result, we can conclude that for such measures, $\dim_H \mu =0$.

\section{Packing dimension} \label{sec:pack}

In the previous section we completely determined the Hausdorff dimension of the measures of our interest. Now we proceed to compute the packing dimension. First we give a lower bound for the upper local dimension. The proof uses similar ideas to the proof of Lemma \ref{upperbound}: we choose a particular cover of the ball and use that the Birkhoff sums for the potentials $-\log p_{a_1}, -\log r_{a_1}$ grow faster than linear.

\begin{prop}
Suppose $T$ is a Gauss-like map and $\mu$ is an infinite entropy Gibbs measure satisfying assumption \ref{assumption}. Then 
\begin{align*}
\limsup_{r\to 0} \dfrac{\log \mu(B(x,r))}{\log r} \geq s
\end{align*}
for $\mu$ almost every $x\in I$.
\end{prop}

\begin{proof} By Birkhoff's Ergodic Theorem, there exists a subset $Z_1\subset I$ of full measure such that
\begin{align*}
    \lim_{n\to\infty} \dfrac{f_{n,1}}{n} = p_1 
\end{align*}
in $Z$, where $f_{n,k}$ is as defined in the proof of \eqref{markovexact}. Intersect $Z_1$ with the subset $Z_2\subset I$ of full measure given by Lemma \ref{infinitelyap} and pick a point $x\in Z_1\cap Z_2$. Since $p_1<1$, we can pick a subsequence $k_n\nearrow \infty$ such that $a_{k_n}\neq 1$ for every $n$. Then, for all $n$, take $r_n=\min\{ |I_{k_n}|, |I_{k_n}^r|, |I_{k_n}^\ell|\}=|I_{k_n}^\ell|$. Here we denote $I_n^\ell=I(a_1,\hdots,a_{n-1},a_n+1)$ and $I_n^r=I(a_1,\hdots,a_{n-1},a_n-1)$ whenever $a_n>1$. This choice of $r_n$ implies that $B(x,r_n)\subseteq I_{k_n}^\ell\cup I_{k_n}\cup I_{k_n}^r$. From the Gibbs property and the fact that $\varphi(x_n),\varphi(x_{n+1})$, and $r_n,r_{n+1}$ are comparable, it follows that there are constants $C_1,C_2>0$ such that $\mu (I_n^\ell\cup I_n\cup I_n^r) \leq C_1\mu (I_n )$ and $|I_n^\ell|\geq C_2 |I_n|$ for every $n$. Using this and Lemma \ref{metricprop} we have that
\begin{align*}
\dfrac{\log \mu (B(x,r_n ))}{\log r_n} &\geq \dfrac{\log(  C_1\mu (I_{k_n} ))}{\log(C_2 |I_{k_n}|)}   \\
&\geq \dfrac{\log C_1 +k_n G_1 + G_2 \sum_{i=1}^{k_n}\log p_{a_i}}{\log C_2-D_2 - k_n\log D_1 +\sum_{i=1}^{k_n}\log r_{a_i}}.
\end{align*}
By Lemma \ref{infinitelyap} and Theorem \ref{markovexact}, the last expression converges to $s$, as desired.
\end{proof}

Giving an upper bound for the upper local dimension requires a more involved analysis of the geometric structure of the partition and its relation to the geometry of the balls. We will need the following lemma:

\begin{lema}\label{ineqsums}
Suppose that $\{ r_n\}$ decays polynomially with degree $\alpha >1$. Then, for every $0<\delta<\min\{1/3,(\alpha-1)/(\alpha+1)\}$, $0<\eta<1/2$ there exists $k_0\in\N$ such that
\begin{align*}
    \dfrac{\log\displaystyle\sum_{m=k}^{n+k}p_m}{\log\displaystyle\sum_{m=k-1}^{n+k+1}r_m} \leq \dfrac{1+\delta}{\alpha-\delta}+\eta
\end{align*}
for all $k\geq k_0$ and $n\in\N$.
\end{lema}

\begin{proof}
Recall that for such sequence $\{r_n\}$, we have that $s=1/\alpha$. Fix $0<\delta<\min\{1/3,s(\alpha-1)/(\alpha+1)\}$, $0<\eta<1/2$. Note that this implies that
\begin{align*}
    \dfrac{\delta}{\alpha-1-\delta}  < s = \dfrac{1}{\alpha} < \dfrac{1+\delta}{\alpha-\delta}.
\end{align*}

Now, since
\begin{align*}
    \lim_{k\to\infty}\dfrac{(1+\delta)\log 2 + \delta\log k}{\log(\alpha-1-\delta)+(\alpha-1-\delta)\log(k-2)} = \dfrac{\delta}{\alpha-1-\delta} < \dfrac{1+\delta}{\alpha-\delta}
\end{align*}
and
\begin{align*}
    \lim_{k\to\infty}\dfrac{(1+\delta)\log(2k)}{(\alpha-\delta)\log(k-1)-\log 3} = \dfrac{1+\delta}{\alpha-\delta},
\end{align*}
we can find $k_0\in\N$ such that
\begin{align*}
    \dfrac{(1+\delta)\log 2 + \delta\log k}{\log(\alpha-1-\delta)+(\alpha-1-\delta)\log(k-2)} &< \dfrac{1+\delta}{\alpha-\delta}+\eta \\
    \dfrac{(1+\delta)\log(2k)}{(\alpha-\delta)\log(k-1)-\log 3}&<\dfrac{1+\delta}{\alpha-\delta}+\eta
\end{align*}
for all $k\geq k_0$.  It can be proved using calculus that for $\delta<(\alpha-1)/2$, the inequality 
\begin{align*}
    (1+\delta)\log(2k) \leq (\alpha-\delta)\log(k-1)-\log 3
\end{align*}
holds for sufficiently large $k$, so we can take $k_0$ large enough so that this holds. Finally, we can take $k_0$ large enough so that we also have
\begin{align*}
     &r_k \leq \dfrac{1}{k^{\alpha -\delta}} \\
    &\dfrac{1}{k^{1+\delta}} \leq p_k 
\end{align*}
for all $k\geq k_0$. Let $n\in\N$. We divide in two cases:

\textbf{Case 1:} $n\geq k$. Then
\begin{align*}
    \sum_{m=k}^{n+k} p_m \geq \dfrac{n}{(2k)^{1+\delta}} \geq \dfrac{1}{2^{1+\delta}k^\delta}
\end{align*}
and
\begin{align*}
    \sum_{k-1}^{n+k+1} r_m \leq \sum_{m=k-1}^{n+k+1} \dfrac{1}{m^{\alpha-\delta}} \leq \sum_{m=k-1}^\infty \dfrac{1}{m^{\alpha-\delta}} \leq \dfrac{1}{\alpha-1-\delta}\left(\dfrac{1}{(k-2)^{\alpha-1-\delta}}\right)
\end{align*}
for all $k\geq k_0$. Then
\begin{align*}
    \dfrac{\log\displaystyle\sum_{m=k}^{n+k}p_m}{\log\displaystyle\sum_{m=k-1}^{n+k+1}r_m} &\leq \dfrac{(1+\delta)\log 2 + \delta\log k}{\log(\alpha-1-\delta)+(\alpha-1-\delta)\log(k-2)}  \leq \dfrac{1+\delta}{\alpha-\delta}+\eta
\end{align*}
for all $k\geq k_0$.

\textbf{Case 2:} $n<k$. Then
\begin{align*}
    \sum_{m=k}^{n+k} p_m \geq \dfrac{n+1}{(2k)^{1+\delta}} 
\end{align*}
and
\begin{align*}
    \sum_{k-1}^{n+k+1} r_m \leq \sum_{m=k-1}^{n+k+1} \dfrac{1}{m^{\alpha-\delta}} \leq \dfrac{n+3}{(k-1)^{\alpha-\delta}} \leq 3\dfrac{(n+1)}{(k-1)^{\alpha-\delta}}.
\end{align*}
Hence
\begin{align*}
    \dfrac{\log\displaystyle\sum_{m=k}^{n+k}p_m}{\log\displaystyle\sum_{m=k-1}^{n+k+1}r_m} &\leq \dfrac{(1+\delta)\log(2k)-\log(n+1)}{(\alpha-\delta)\log(k-1)-\log 3 - \log(n+1)}.
\end{align*}
We use the following Lemma:
\begin{lema}
For $a,b,c>0$ such that $a-c,b-c>0$, we have that
\begin{align*}
    \dfrac{a-c}{b-c}\leq \dfrac{a}{b}
\end{align*}
if and only if $b\geq a$.
\end{lema}
We can use this with $a=(1+\delta)\log(2k)$, $b=(\alpha-\delta)\log(k-1)-\log 3$ and $c=\log(n+1)$. This implies that
\begin{align*}
    \dfrac{\log\displaystyle\sum_{m=k}^{n+k}p_m}{\log\displaystyle\sum_{m=k-1}^{n+k+1}r_m} &\leq \dfrac{(1+\delta)\log(2k)}{(\alpha-\delta)\log(k-1)-\log 3} \leq \dfrac{1+\delta}{\alpha-\delta}+\eta.
\end{align*}
for all $k\geq k_0$, as we wanted to prove.
\end{proof}

With the previous lemma, we can now prove the upper bound for the upper local dimension. The proof is based on carefully choosing the covers of the balls; such covers must be fine enough so they are not affected by Proposition \ref{limsup}. This means that we want to cover the ball with cylinders of the same scale, otherwise, the cover would yield trivial bounds.

\begin{prop}
Suppose $T$ is a Gauss-like map and $\mu$ is an infinite entropy Gibbs measure satisfying assumption \ref{assumption}. Then 
\begin{align*}
\limsup_{r\to 0} \dfrac{\log \mu(B(x,r))}{\log r} \leq s
\end{align*}
for $\mu$ almost every $x\in I$.
\end{prop}

\begin{proof}
Let $x$ be a point where Theorem \ref{markovexact} and Lemma \ref{infinitelyap} applied to $f=-\log r_{a_1}$ hold. 
Given $r>0$, there exists a unique natural number $n=n(r)$ such that
\begin{align*}
    |I_n(x)| < r \leq |I_{n-1}(x)|.
\end{align*}
Note that $n\to\infty$ as $r\to 0$. Let $\delta>0$ and $\eta$ as in Lemma \ref{ineqsums}. Then there exists $k_0\in\N$ such that
\begin{align*}
    \dfrac{\log\displaystyle\sum_{m=k}^{n+k}p_m}{\log\displaystyle\sum_{m=k-1}^{n+k+1}r_m} \leq \dfrac{1+\delta}{\alpha-\delta}+\eta
\end{align*}
for all $k\geq k_0$. Recall that by $I_n^{m\cdot r}(x)$ we denote the cylinder $I(a_1,\hdots,a_{n-1},a_n-m)$, where $(a_n)$ is the sequence coding $x$ and $m<a_n$. We separate the proof in two cases:

\textbf{Case 1:} 
\begin{align*}
    I(a_1,\hdots,a_{n-1},k_0) \subset B(x,r).
\end{align*}
In this case, using Lemma \ref{metricprop} we have that 
\begin{align*}
    \log(\mu(B(x,r))) &\geq \log(\mu(I(a_1,\hdots,a_{n-1},k_0))) \\
    & \geq \sum_{k=1}^{n-1} \log p_{a_k}+\log p_{k_0} -nG_1 -G_2.
\end{align*}
We get then
\begin{align}
    \dfrac{\log{\mu(B(x,r))}}{\log r}\leq \dfrac{\sum_{k=1}^{n-1} \log p_{a_k}+\log p_{k_0} -nG_1 -G_2}{\sum_{k=1}^{n-1} \log r_{a_k} +nD_1 +D_2} \nonumber \\
    \leq \dfrac{(s+\delta)\sum_{k=1}^{n-1} \log r_{a_k}+\log p_{k_0} -nG_1 -G_2}{\sum_{k=1}^{n-1} \log r_{a_k} +nD_1 +D_2} . \label{pack1}
\end{align}

\textbf{Case 2:} 
\begin{align*}
    I(a_1,\hdots,a_{n-1},k_0) \not\subset B(x,r).
\end{align*}
This implies that there exists $k_1\in\N$ such that
\begin{align*}
    \bigcup_{m=1}^{k_1-1} I_n^{k\cdot r}(x) &\subset B(x,r) , \\
    \Bigg| \bigcup_{m=0}^{k_1} I_n^{k\cdot r}(x) \Bigg| &> r
\end{align*}
and consequently
\begin{align*}
    \log(\mu(B(x,r))) &\geq \sum_{k=1}^{n-1}\log p_{a_k} + \log\left(\sum_{k=1}^{k_1-1} p_{a_n-k}\right) - nG_1 - G_2 \\
    \log r &\leq \sum_{k=1}^{n-1}\log r_{a_k} + \log\left(\sum_{k=0}^{k_1} r_{a_n-k}\right) + nD_1 + D_2.
\end{align*}
We obtain then
\begin{align*}
    \dfrac{\log(\mu(B(x,r)))}{\log r} \leq \dfrac{\sum_{k=1}^{n-1}\log p_{a_k} + \log\left(\sum_{k=1}^{k_1-1} p_{a_n-k}\right) - nG_1 - G_2}{\sum_{k=1}^{n-1}\log r_{a_k} + \log\left(\sum_{k=0}^{k_1} r_{a_n-k}\right) + nD_1 + D_2}.
\end{align*}
Using inequality \eqref{ineqsums}
\begin{align*}
    \dfrac{\log(\mu(B(x,r)))}{\log r} \leq \dfrac{\sum_{k=1}^{n-1}\log p_{a_k} + \left(\frac{1+\delta}{\alpha-\delta}+\eta \right)\log\left(\sum_{k=0}^{k_1} r_{a_n-k}\right) - nG_1 - G_2}{\sum_{k=1}^{n-1}\log r_{a_k} + \log\left(\sum_{k=0}^{k_1} r_{a_n-k}\right) + nD_1 + D_2}.
\end{align*}
For $\delta > 0$, there exist $n_0\in\N$ such that 
\begin{align*}
     \dfrac{-\sum_{k=1}^{n-1} \log q_{a_k}}{-\sum_{k=1}^{n-1} \log r_{a_k}}< s + \delta 
\end{align*}
for all $n\geq n_0$. We obtain
\begin{align}
    &\dfrac{\log(\mu(B(x,r)))}{\log r} \leq \dfrac{(s+\delta)\sum_{k=1}^{n-1}\log r_{a_k} + \left(\frac{1+\delta}{\alpha-\delta}+\eta \right)\log\left(\sum_{k=0}^{k_1} r_{a_n-k}\right) - nG_1 - G_2}{\sum_{k=1}^{n-1}\log r_{a_k} + \log\left(\sum_{k=0}^{k_1} r_{a_n-k}\right) + nD_1 + D_2} \nonumber\\
    &\leq \max\Bigg\{(s+\delta),\left(\frac{1+\delta}{\alpha-\delta}+\eta \right) \Bigg\}\cdot\dfrac{\sum_{k=1}^{n-1}\log r_{a_k} + \log\left(\sum_{k=0}^{k_1} r_{a_n-k}\right) - nG_1 - G_2}{\sum_{k=1}^{n-1}\log r_{a_k} + \log\left(\sum_{k=0}^{k_1} r_{a_n-k}\right) + nD_1 + D_2}.\label{pack2}
\end{align}
By Lemma \ref{infinitelyap} we have that the right hand side of \eqref{pack1} and \eqref{pack2} converge to
\begin{align*}
    (s+\delta) \ , \ \max\Bigg\{(s+\delta),\left(\frac{1+\delta}{\alpha-\delta}+\eta \right) \Bigg\}
\end{align*}
respectively. We conclude that
\begin{align*}
    \limsup_{r\to 0} \dfrac{\log(\mu(B(x,r)))}{\log r}\leq \max\Bigg\{(s+\delta),\left(\frac{1+\delta}{\alpha-\delta}+\eta \right) \Bigg\}.
\end{align*}
Letting $\delta\to 0$ and $\eta \to 0$, we obtain the desired result.
\end{proof}

\begin{coro}
For an infinite entropy Gibbs measure $\mu$ satisfying \eqref{conditions}, associated to a Gauss-like map, we have that $0=\underline{d}(x)<s= \overline{d}(x)$ for almost every point, and hence $\mu$ is not exact dimensional.
\end{coro}

With this we have found the almost sure behavior of the local dimensions, and hence, we have obtained values for both the packing and the Hausdorff dimension.

\section{Final remarks}

Theorem \ref{mainteo} implies that for maps such that $\{r_n\}$  decays polynomially, the Hausdorff dimension of ergodic invriant measures with infinite entropy is equal to zero under mild independence and regularity assumptions on the measure. 

\begin{que}
Is there an ergodic invariant measure $\mu$ for a Gauss-like map with $h_\mu=\lambda_\mu=\infty$, $r_n\neq p_n$ and $\dim_H \mu >0$?
\end{que}

The condition $r_n\neq p_n$ rules out the Lebesgue measure, which clearly has dimension equal to $1$. We can construct such measure if we assume that $r_n$ decays slower than polynomial.

We also formulate two questions for a more general case:

\begin{que}\label{Q2}
What can be said about the almost sure value of the symbolic dimension when $\mu$ is only assumed to be ergodic?
\end{que}

\begin{que}\label{Q3}
What can be said about $\dim_H \mu$ when $\mu$ is only assumed to be ergodic?
\end{que}

The main difficulty with question \ref{Q2} is that our methods rely on the asymptotic independence of the digits in the symbolic space. This implies that we can write the measure and diameter of cylinders in the form of Birkhoff sums, allowing us to use ergodic theoretic methods to study the almost sure behavior of such sums.

On the other hand, the main difficulty with question \ref{Q3} is that one of the main ergodic theoretic tools we use (Theorem \ref{aarnak}) assumes the process $\{X_n\}$ is CF-mixing. For measures which do not satisfy any kind of independence assumption, we are not able to use such techniques.

\subsection*{Acknowledgements} The author would like to thank his advisor T. Jordan for suggesting the problem and all the valuable suggestions. The author would also like to thank G. Iommi, C. Lutsko and M. Todd for their valuable comments about the first draft of this article. The research leading to these results was partially supported the Becas Chile scholarship scheme from CONICYT.

\bibliographystyle{alpha}
\bibliography{infinite}

\end{document}